\documentclass[12pt]{article}
\usepackage[utf8]{inputenc}
\usepackage[T2A]{fontenc}
\usepackage{cmap}
\usepackage{indentfirst}
\usepackage[english]{babel}
\usepackage{amsmath,amssymb,amsthm,graphicx}
\usepackage{bbm}
\usepackage[usenames]{color}
\usepackage{colortbl}
\usepackage{caption}
\usepackage{subcaption}
\usepackage[unicode, pdftex]{hyperref}
\usepackage[left=2.5cm,right=2.2cm,top=2.5cm,bottom=3.5cm]{geometry}
\usepackage{wasysym}

\newcommand\R{\mathbb R}
\newcommand\Z{\mathbb Z}

\newtheorem{theorem}{Theorem}[section]

\newtheorem{ex}{Example}[section]
\newtheorem{proposition}{Proposition}[section]
\newtheorem{lemma}{Lemma}[section]
\newtheorem{conj}{Conjecture}[section]
\newtheorem{definition}{Definition}[section]
\newtheorem{notation}{Notation}[section]
\newtheorem{Corollary}{Corollary}[section]

\begin{document}

\title{New estimates for $d_{2,1}$ and $d_{3,2}$ }
\author{Arkadiy Aliev}
\date{}
\maketitle
\begin{abstract}
     Let $K$ be a convex body in $\R^{n}$. Let $ d_{n,n-1}(K)$ be the smallest possible density of a non-separable lattice of translates of $K$. In this paper we prove the estimate $d_{2,1}(K)\leq \frac{\pi\sqrt{3}}{8}$ for $K\subset \R^{2}$, with equality if and only if $K$ is an ellipse, which was conjectured by E. Makai. Also we prove the estimate $d_{3,2}(K)\leq\frac{\pi}{4\sqrt{3}}$ for $K\subset\R^{3}$ using projection bodies. 
    
\end{abstract}

\begin{section}{Preliminaries}
    A set $K\subset\R^{n}$ is a convex body if it is convex, compact, and its interior $int K$ is non-empty. We denote the volume of K by $|K|$. 
    The difference body of $K$ is defined as $K-K:=\{x-y|x,y\in K\}$.  For any $\lambda \in \R$ we define $\lambda K := \{\lambda x | x\in K\}$. $K$ is centrally symmetric if $K = -K$.

    \begin{definition}
        If  $0\in \text{int}K$ then the polar body of $K$ is 
        $$
            K^{\circ}:=\{x\in\R^{n} |\forall y\in K\text{ } \langle y,x \rangle\leq 1 \}.
        $$
    \end{definition}

    \begin{proposition}[Blaschke-Santaló inequality]
        For a centrally symmetric  convex body $K\subset \R^{n}$ we have 
        $$
            |K||K^{\circ}|\leq |B|^2, \text{ where $B$ is the unit ball in $\R^{n}$. } 
        $$
    \end{proposition}
    \begin{definition}
    For a convex body $K$ define the projection body of $K$  by its support function:  
    $$
    h_{\Pi K}(u) = |Pr_{u^{\perp}}K| \text{ for all $u\in S^{n-1}$.}
    $$     
    It is well-known that the definition is correct and it defines a convex body $\Pi K.$
\end{definition}

\begin{proposition}[Petty's inequality]
    Let $B$ be the unit ball in $\R^{n}$. Then for a convex body $K$ we have:
    $$
    |(\Pi K)^{\circ}| |K|^{n-1}\leq |(\Pi B)^{\circ}| |B|^{n-1}.
    $$
\end{proposition}
    We will also need some definitions from the geometry of numbers. 
    \begin{definition}
        $\Lambda\subset \R^{n}$ is a lattice if $\Lambda = A\mathbb{Z}^n$ for some $A\in GL(n,\mathbb{R})$. We also define $d(\Lambda) := |det A|.$
    \end{definition}

    \begin{definition}
        A lattice $\Lambda$ is called $K$-admissible if $ \Lambda \cap int K = \{0\}$.    
    \end{definition}

    \begin{definition}
        For a convex body $K$ its critical determinant is defined as  $$\Delta(K):=\min\{d(\Lambda)| \Lambda \text{ is } K\text{-admissible} \}$$ 
    \end{definition}

    \begin{definition}
        A lattice $\Lambda$ is called $K$-critical if $\Lambda$ is $K$-admissible and $d(\Lambda)=\Delta(K)$.     
    \end{definition}

    \begin{definition}
        Let $\Lambda$ be a lattice, then a lattice of translates of $K$ is defined as $$\Lambda + K = \{x+y| x\in \Lambda,y\in K\}.$$  
    \end{definition}
    
    \begin{definition}
        Let $\Lambda + K$ be a lattice of translates of $K$. Then its density is defined as $$D(\Lambda, K) := \frac{|K|}{d(\Lambda)}.$$      
    \end{definition}

    \begin{definition}
        We denote by $\delta_{L}(K)$ the density of the densest lattice packing in $\R^{n}$ by translates of $K$. It is easy to see that for centrally symmetric $K$ we have:
        $$
            \delta_{L}(K) = \frac{|K|}{2^{n}\Delta(K)}.
        $$
    \end{definition}

    \begin{definition}
        A lattice of translates of $K$ is non-separable if each affine $(n-1)$-subspace in $\R^{n}$ meets $x+K$ for some $x\in\Lambda$. 
    \end{definition}

    \begin{definition}
        For a convex body $K$ its *critical determinant is defined as  $$\Delta^{*}(K):=\max\{d(\Lambda)| \Lambda+K \text{ is non-separable} \}$$ 
    \end{definition}

    \begin{definition}
        A lattice $\Lambda$ is called $K$-*critical if $\Lambda + K$ is non-separable and $d(\Lambda)=\Delta^{*}(K)$.
    \end{definition}
    
    \begin{definition}
        $d_{n,n-1}(K) = \min\{d(\Lambda,K)| \Lambda+K \text{ is non-separable} \} = \frac{|K|}{\Delta^{*}(K)}$
    \end{definition}

    \begin{definition}
        Lattice width $\omega_{K}:\Z^{n}\setminus{0}\to \R$ is defined as 
        $$
        \omega_{K}(u)=\max_{x,y\in K}\langle u, x-y\rangle.
        $$
    \end{definition}

\end{section}
    
\begin{section}{Introduction}

     The Reinhardt Conjecture is an open long-standing problem about finding a centrally symmetric body $K\subset \R^{2}$ with the smallest possible value of $\delta_{L}(K)$. 
     Reinhardt conjectured that the unique solution up to an affine transformation is the smoothed octagon (an octagon rounded at corners by arcs of hyperbolas) and the conjectured minimum of $\delta_{L}(K)$ is $\frac{8-\sqrt{32}-\ln2}{\sqrt{8}-1}\approx 0.902414,$ while the best known estimation is $\delta_{L}(K)\geq  0.8926... $ by Tammela \cite{6}. 
     Comprehensive information about this conjecture can be found in \cite{5}.

    It is well known that the dual problem to the Reinhardt conjecture is the question about upper bounds for $d_{2,1}(K)$.  This problem was considered by Endre Makai Jr. in  \cite{0}, \cite{1}. He conjectured that $d_{2,1}(K)\leq\frac{\pi\sqrt{3}}{8}$ with equality only for ellipses. In his works Endre Makai Jr. obtained the following dual property and nearly accurate estimate.    
     \begin{ex}
        For the unit ball $B$ in $\R^{2}$ we have $d_{2,1}(B) = \frac{\pi\sqrt{3}}{8}$ and $\delta_{L}(B) = \frac{\pi}{2\sqrt{3}}\approx 0.90689$.
    \end{ex}
    
    \begin{ex}
        For a triangle $T$ we have $d_{2,1}(T) = \frac{3}{8}$. Let $T = (0,0)(1,1/2)(1/2,1).$ Then  $\mathbb{Z}^{2}$ is the *critical lattice.
    \end{ex}
    
    \begin{proposition}[Duality condition, \cite{1}]
    \label{prop2.1}
        For a convex body $K\subset\R^{n}$ we have:
        $$
            d_{n,n-1}(K) = \frac{|K||((K-K)/2)^{\circ}|}{4^{n}\delta_{L}(((K-K)/2)^{\circ})}.
        $$
    \end{proposition}
   
    \begin{proposition}
         \cite{0}
         For a convex body $K\subset\R^{2}$ we have $d_{2,1}(K) \leq 0.6910...$.
    \end{proposition}

    \begin{conj}[Endre Makai Jr., \cite{0}] 
    \label{conj}
         $ \max\{ d_{2,1}(K) | K \subset \R^{2}\text{ is a convex body}\}= \frac{\pi\sqrt{3}}{8}= 0.68017...$, with equality possible only for ellipses. 
    \end{conj}
    
    From Proposition \ref{prop2.1} we get the idea  of using estimation of $\delta_{L}(K)$ to get upper bounds for $d_{2,1}(K)$. But unfortunately, accurate estimates in $\R^{2}$ cannot be obtained in that way and the main reason is that the optimal body for $\delta_{L}$ optimisation is not an ellipse. 
    Nevertheless, the estimate obtained in \cite{0} is very close to the conjectured value.  

    This paper will provides a proof of Conjecture  \ref{conj}. 
    \begin{theorem}
        \label{1.1}
        $ \max\{ d_{2,1}(K) | K \subset \R^{2}\text{ is a convex body}\}= \frac{\pi\sqrt{3}}{8}= 0.68017...$. In addition, if $d_{2,1}(K)=\frac{\pi\sqrt{3}}{8}$, then $K$ is an ellipse.   
    \end{theorem}

    \begin{Corollary} 
    \label{c2.1}
        For each centrally symmetric convex body $K\subset \R^{2}$ we have $\delta_{L}(K)\geq \frac{1}{2\pi\sqrt{3}}|K||K^{\circ}|$, with equality if and only if $K$ is an ellipse.  
    \end{Corollary}
    The Ulam's packing conjecture states that for a centrally symmetric convex body $K\subset \R^{3}$ we have $\delta_{L}(K) \geq \frac{\pi}{\sqrt{18}}$. If this statement holds, then by using duality we could obtain the sharp estimate $d_{3,2}(K) \leq \frac{\pi}{6\sqrt{2}},$ with equality only for ellipsoids. Nevertheless, the best known estimate for a packing constant in $\R^{3}$ is  $\delta_{L}(K) \geq 0.53835...$ by E.H. Smith \cite{7}.
     Therefore we can use duality to get the estimate:
     $$
     d_{3,2}(K) =  \frac{|K||K^{\circ}|}{64\delta_{L}(K^{\circ})}\leq \frac{(\frac{4\pi}{3})^{2}}{64\cdot 0.53835...} = 0.509251...
     $$

    This paper provides an idea of using projection bodies to improve $d_{3,2}$ estimates without using known bounds for $\delta_{L}$. We prove the following theorem. 

    \begin{theorem}
    \label{th2.2}
        For a convex body $K\subset\R^{3}$ we have $d_{3,2}(K)\leq \frac{\pi}{4\sqrt{3}} = 0.453449...$.
    \end{theorem}

    \begin{ex}
        For the unit ball $B\subset\R^{3}$ we have $d_{3,2}(B)=\frac{\pi}{6\sqrt{2}} = 0.37024...$.
    \end{ex}
\end{section}

\begin{section}{Сalculations}
    
    \begin{notation}
        In this section we assume that $K\subset \mathbb{R}^2$ is strongly convex and has $C^{\infty}$ boundary. 
    \end{notation}

    \begin{notation}
        The support function of $K$ $h:S^{1}\to\mathbb{R}$ is defined as $h(\theta) = \max_{x\in K}\langle (\cos\theta, \sin\theta),x\rangle$.
    \end{notation}

    \begin{notation}
        Let $a = (a_{1},a_{2}), b = (b_{1},b_{2})$. Then we denote 
        $$
        det(a,b):=
        det\left(
                \begin{array}{cc}
                   a_{1} &  a_{2}\\
                   b_{1} &  b_{2}
                \end{array}
            \right)  
        $$
    \end{notation}
    \noindent The following statements are well-known.
    \begin{proposition}
        Parameterization in terms of the support function: consider  
        $$
        \gamma(x) = 
            \left(
                \begin{array}{cc}
                    cos(x) & -sin(x)\\
                    sin(x) & cos(x)
                \end{array}
            \right)
            \left(
                \begin{array}{c}
                    h(x)  \\
                    h'(x)  
                \end{array}
            \right).
        $$
        Then
        \begin{itemize}
            \item $\gamma(x)\in K$
            \item $\langle (\cos x, \sin x),\gamma(x)\rangle= \max_{y\in K}\langle (\cos x, \sin x),y\rangle$
            \item $\gamma'(x) = (h(x)+h''(x))(-sin(x),cos(x))$
            \item $det(\gamma(x),\gamma'(x)) = h(x)(h(x)+h''(x))$
            \item $det(\gamma'(x),\gamma''(x)) = (h(x)+h''(x))^2$
        \end{itemize}
    \end{proposition}

    \begin{proposition} We list some properties of the support function:  
        \begin{itemize}
            \item $h+h''\geq 0$ 
            \item $|K| = \frac{1}{2} \int_{0}^{2\pi} h^2 + hh'' = \frac{1}{2} \int_{0}^{2\pi} h^2 - h'^2$
            \item $|K^{\circ}| =  \frac{1}{2} \int_{0}^{2\pi} \frac{1}{h^2}$
        \end{itemize}
    \end{proposition}

    \begin{proposition}
    \label{2.3}
        Let $\Gamma(t)=\gamma(t)+l(t)\frac{\gamma'(t)}{|\gamma'(t)|}$, where $\gamma:\mathbb{R}\to\mathbb{R}^{2}$ and $l:\R\to\R$ are smooth periodic functions with period $2\pi$. We also assume that $|\gamma'|>0$.
        Then 
        $$
            \frac{1}{2}\int_{0}^{2\pi} det(\Gamma, \Gamma') - \frac{1}{2}\int_{0}^{2\pi} det (\gamma, \gamma') = \frac{1}{2}\int_{0}^{2\pi}\frac{l^2}{|\gamma'|^2}det(\gamma',\gamma'').
        $$
    \end{proposition}
    
    \begin{proof}
        $$
            \begin{aligned}
                  \int_{0}^{2\pi} det(\Gamma,\Gamma') = \int_{0}^{2\pi} det(\gamma(t)+l(t)\frac{\gamma'(t)}{|\gamma'(t)|},\gamma'(t)+\left(l(t)\frac{\gamma'(t)}{|\gamma'(t)|}\right) ')  \\
                  = \int_{0}^{2\pi} det(\gamma,\gamma') + \int_{0}^{2\pi} det(\gamma,\left(l(t)\frac{\gamma'(t)}{|\gamma'(t)|}\right)') + \int_{0}^{2\pi} det(l(t)\frac{\gamma'(t)}{|\gamma'(t)|},\left(l(t)\frac{\gamma'(t)}{|\gamma'(t)|}\right)' )\\
                  = \int_{0}^{2\pi} det(\gamma,\gamma') - \int_{0}^{2\pi} det(\gamma',l(t)\frac{\gamma'(t)}{|\gamma'(t)|}) + \int_{0}^{2\pi}\frac{l^2}{|\gamma'|^2}det(\gamma',\gamma'') \\
                  = \int_{0}^{2\pi} det(\gamma,\gamma') +  \int_{0}^{2\pi}\frac{l^2}{|\gamma'|^2}det(\gamma',\gamma'')
            \end{aligned}
        $$
    \end{proof}

\end{section}

\begin{section}{ Proof of the Theorem 2.1 (inequality part)}

    \begin{lemma}\cite{1}
    \label{3.1}        
        $$
            \begin{aligned}
                \max\{ d_{2,1}(K) | K \subset \R^{2}
                \text{ is a convex body}\} = \\
                \max\{d_{2,1}(K) | K \subset \R^{2}
                \text{is a centrally symmetric convex body}\}
            \end{aligned}
        $$
    \end{lemma}
    \noindent Therefore, we can assume that $K$ is centrally-symmetric.

    \begin{notation}
        A triangle is a central triangle if its barycenter is in the origin. 
    \end{notation}

    \begin{lemma}
    \label{4.2}
        Let $K$ be a centrally symmetric convex body and suppose that there exists a central triangle $T^{\circ}$ whose vertices belong to $\partial K^{\circ}$, such that $|K||T^{\circ}| \leq \frac{3\sqrt{3}}{4}\pi$. Then $d_{2,1}(K)\leq \frac{\sqrt{3}\pi}{8}.$    
    \end{lemma}
    \begin{figure}[h]
        \centering
        
        \includegraphics[scale=0.5]{"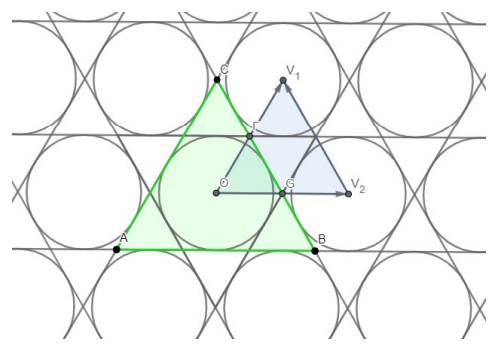"}  
        \caption{
             Non-separable lattice for a convex body $K$.
        }
        \label{fig:nsl}
    \end{figure}
    \begin{proof}
         $(T^{\circ})^{\circ} = T = ABC$ is a central triangle circumscribed to $K$. $|T||T^{\circ}|=\frac{27}{4}$, hence $\frac{|K|}{|T|}\leq \frac{\pi}{3\sqrt{3}}$.  Further, we can generate a non-separable lattice $\Lambda$ with $d(\Lambda) = \frac{8}{9}|T|$ by $V_{1}=\frac{2}{3}(C-A)$ and $V_{2}=\frac{2}{3}(B-A)$ as in Figure \ref{fig:nsl}.

    \end{proof}
    \begin{figure}[h]
    \centering
    \includegraphics[scale=0.7]{"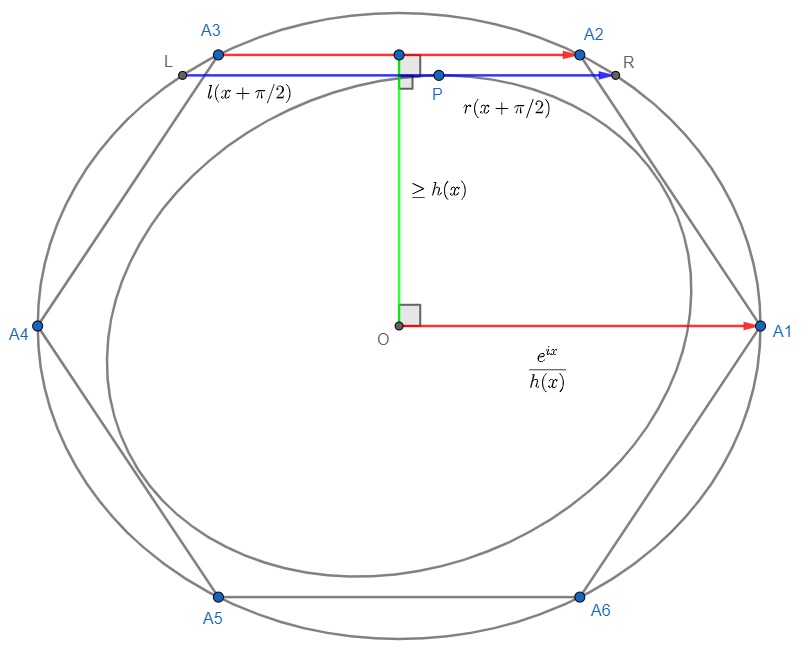"}  
    \caption{
        The outer body is $K^{\circ}$, the inner body is  $iK$; the point $P$ is $\gamma(x+\pi/2)$, $LR$ is the tangent line to $iK$ at the point $P$, hence we have $|LP| = l(x+\pi/2)$, $|PR| = r(x+\pi/2)$.
    }
    \end{figure}
  
    \begin{lemma}
    \label{3.3}
        Assume that $K$ is centrally symmetric, strictly convex and $C^{\infty}$
         boundary and let the minimal area of a central triangle inscribed in $K^{\circ}$ be equal to $\frac{3}{2}$. Then we have $\frac{3}{4}|K^{\circ}|\geq |K|$. 
    \end{lemma} 
    
    \begin{proof}
       For each central triangle $T$ inscribed in $K^{\circ }$ consider an inscribed affine regular hexagon $conv(T\cup -T)$ and let $A_{1}..A_{6}$ be one of them. We put  $A_{1}=\frac{e^{ix}}{h(x)} = A_{2}-A_{3}$. 
        
        From the minimum area condition we have $|A_{2}A_{3}O|\geq \frac{1}{2}$, hence the distance $d(O,A_{2}A_{3})\geq\frac{1}{|A_{2}A_{3}|} = h(x)$. Further, let $A$ be the intersection of all inscribed in $K^{\circ}$ affine regular hexagons. Then we have $iK\subset A \subset K^{\circ}$, where $iK$ is  $\frac{\pi}{2}$-rotated $K$. 
        
        Let $\gamma$ parametrize $iK$ in terms of the support function and let $L(x)$ be the tangent line to $iK$ at the point $\gamma(x)$.
        Further, define $l(x)$ and $r(x)$ as 
        lengths of left and right parts of $|L(x)\cap K^{\circ}|$ with respect to  $\gamma(x)$. Since $iK\subset A$ and since $K^{\circ}$ is centrally symmetric we have  $r(x)+l(x)\geq \frac{1}{h(x-\pi/2)}$. Then using Proposition \ref{2.3} we get: 
        \begin{align*}
            |K^{\circ}| - |K| = |K^{\circ}| - |iK| = \frac{1}{2}\int_{0}^{2\pi}l^2 = \frac{1}{2}\int_{0}^{2\pi}r^2. 
        \end{align*}
        Therefore: 
        \begin{align*}
            2(|K^{\circ}| - |K|) = \frac{1}{2}\int_{0}^{2\pi}(l^2+r^2)\geq \frac{1}{4}\int_{0}^{2\pi}(l+r)^2 \geq 
            \frac{1}{4}\int_{0}^{2\pi}\frac{1}{h(x-\pi/2)^2} = \frac{1}{2}|K^{\circ}|.
        \end{align*}
        Thus we have $\frac{3}{4}|K^{\circ}|\geq |K|$.
    \end{proof}
    
    \begin{lemma}
    \label{4.4}
        Let $K\subset \R^{2}$ be a centrally symmetric convex body and let the minimal area of a  central triangle inscribed in $K^{\circ}$ be equal to $\frac{3}{2}$. Тhen we have $|K|\leq\frac{\sqrt{3}\pi}{2}$, with equality if and only if $K$ is an ellipse. 
    \end{lemma}
    
    \begin{proof}
        Consider the approximation of $K$ by smooth and strongly convex centrally symmetric bodies $K_{\epsilon}\to K$ and let the minimal area of a central triangle inscribed in $K_{\epsilon}$ be equal to $S_{\epsilon}$.  Then it is obvious that $S_{\epsilon}\to\frac{3}{2}$. Further,  $\sqrt{\frac{2}{3}S_{\epsilon}}K_{\epsilon}$ satisfies the condition of Lemma \ref{3.3}, hence we have $\frac{3}{4}|(\sqrt{\frac{2}{3}S_{\epsilon}}K_{\epsilon})^{\circ}|\geq |\sqrt{\frac{2}{3}S_{\epsilon}}K_{\epsilon}|$. Since the volume is continuous, we have  $\frac{3}{4}|K^{\circ}|\geq |K|$. Finally, using the Blaschke-Santaló inequality we get $\frac{3}{4}\pi^2\geq \frac{3}{4}|K^{\circ}||K|\geq|K|^2$, with equality only for an ellipse.    
    \end{proof}
    \begin{proof}[Proof of Theorem \ref{1.1} (inequality part)]
        Obviously follows from Lemma \ref{3.1}, Lemma \ref{4.2} and Lemma \ref{4.4}.
    \end{proof}
\end{section}

\begin{section}{The equality case}
    \noindent The following Lemmas are well-known. 
    \begin{lemma}
    \label{5.1}
    \cite{4}
        $K+\Z^{n}$ is non-separable if and only if $\omega_{K}(u)\geq 1$ for all $u\in\Z^{n}\setminus{0}.$    
    \end{lemma}

    \begin{lemma}
          Let $K$ be a centrally-symmetric convex body. Then $\Z^{n}+K$ is non-separable if and only if $\Z^{n}$ is $\frac{1}{2}K^{\circ}$-admissible. 
    \end{lemma}

    \begin{proof}
        Obviously follows from Lemma \ref{5.1}.  
    \end{proof}

    \begin{lemma}
        $\Z^{n}$ is *critical for $K$ if and only if $\Z^{n}$ is critical for $\frac{1}{2}K^{\circ}$. 
    \end{lemma}

    \begin{proof}
        Let $\Lambda=A\Z^{2}$ be a $\frac{1}{2}K^{\circ}$-admissible lattice with $d(\Lambda)<1$, that is, $|detA|<1$. Therefore $\Z^{2}$ is $\frac{1}{2}A^{-1}K^{\circ}$-admissible, hence we have that $(A^{*})^{-1}\Z^{2}+K$ is non-separable, but $det(A^{*})^{-1} > 1$, a contradiction.

        Proof in the opposite direction is similar. 
    \end{proof}
    
    \begin{lemma}
    \label{5.3}
        If $K \neq -K$, than $d_{n,n-1}(K) < d_{n,n-1}(K-K)$.     
    \end{lemma}

    \begin{proof}
            From The Brunn–Minkowski inequality we have the estimation $|K|<|(K-K)/2|$.  
            Also it is obvious that $d_{2,1}(K-K)=d_{2,1}((K-K)/2)$ and that $\omega_{K}\equiv \omega_{(K-K)/2}$. Therefore from Lemma \ref{5.1} it follows that for any lattice $\Lambda$ the lattice of translates $K+\Lambda$ is non-separable if and only if $(K-K)/2 +\Lambda$ is non-separable. 
    \end{proof}

    \begin{lemma}\cite{3}
    \label{5.4}
        Let $\Lambda$ be $K$-critical for a centrally-symmetric $K\subset\R^{2}$, and let $C$ be the boundary of $K$. Then one can find three pairs of points $\pm p_{1}, \pm p_{2}, \pm p_{3}$ of the lattice on $C$. Moreover these three points can be
        chosen such that $p_{1} + p_{2} = p_{3}$ 
        and any two vectors among $p_{1}, p_{2}, p_{3}$ form a basis of $\Lambda$.
        
        Conversely, if $p_{1}, p_{2}, p_{3}$ satisfying $p_{1}+p_{2}=p_{3}$ are on C, then the lattice generated by $p_{1}$ and
        $p_{2}$ is $K$-admissible
    \end{lemma}

    \begin{proof}[Proof of Theorem \ref{1.1} (the equality case)]
        Suppose that $d_{2,1}(K) = \frac{\pi\sqrt{3}}{8}$. It follows from Lemma \ref{5.3} that $K$ has to be a centrally symmetric body. Let $\Z^2$ be a critical lattice for $K^{\circ}$, so $|K|=\frac{\sqrt{3}\pi}{2}$. Then, from Lemma \ref{5.4} we conclude that the minimal area of a central triangle inscribed in $K^{\circ}$ is equal to $\frac{3}{2}$. Hence there is an equality in Lemma \ref{4.4}. Thus $K$ is an ellipse.  
    \end{proof}

\end{section}

\begin{section}{Proof of the Theorem 2.2 
}

\begin{proof}[Proof of Theorem \ref{th2.2}.]
    Since $d_{3,2}(K)\leq d_{3,2}(K-K)$, it suffices to сonsider the case of a centrally symmetric $K$. We also assume that $K$ is strongly convex and has $C^{\infty}$ boundary.
    
    Let us construct a lattice  packing for $K^{\circ}$. For $h\in S^{2}$ let the critical lattice for $K^{\circ}\cap h^{\perp}$ corresponds to an affine regular hexagon $A_{1}..A_{6}$ inscribed in $K^{\circ}\cap h^{\perp}$. The horizontal part of of the lattice to be constructed will be generated by the vectors $2A_{1}$ and $2A_{2}$. Denote $$\Lambda_{h^{\perp}}:=\{2A_{1}m+2A_{2}n | n,m\in\ \Z\}.$$
    
    It is easy to see that for any $w \in \Lambda_{h^{\perp}}$ we have $Int(K^{\circ}) \cap Int(K^{\circ}+w) = \emptyset$. Further, define the third generating vector $v := h\cdot2\max_{x\in K^{\circ}}\langle x,h\rangle $. Thus we constructed the lattice $\Lambda := \{ vn + \Lambda_{h^{\perp}} | n\in \Z \}$ and $K^{\circ} + \Lambda $ is obviously a lattice packing of $K^{\circ}$.

    Let $\frac{|K^{\circ}|}{\delta_{L}(K^{\circ})} = 8 \frac{\sqrt{3}\pi}{2}.$ Then $d(\Lambda)\geq 8 \frac{\sqrt{3}\pi}{2}$, thus we have 
    $$
    \frac{1}{8}d(\Lambda) = \frac{1}{4}d(\Lambda_{h^{\perp}})\frac{1}{2}|v| = \Delta(K^{\circ}\cap h^{\perp}) d_{h}\geq \frac{\sqrt{3}\pi}{2}, \text{ where $d_{h}:= \frac{1}{2}|v|.$}
    $$
    
    It is well-known that for $h\in S^{2}$ we have $K^{\circ}\cap h^{\perp} = (Pr_{h^{\perp}}K)^{\circ}$. Therefore by Lemma \ref{4.4} we get $\Delta(K^{\circ}\cap h^{\perp})|Pr_{h^{\perp}}K|\leq\frac{\sqrt{3}\pi}{2}, $ so $d_{h}\geq |Pr_{h^{\perp}}K|.$ 
     Therefore $\Pi K\subset K^{\circ}$ and $ K\subset \Pi^{\circ}K $. Then by using Petty's inequality we obtain the estimate:
    $$
    |K|^3\leq|\Pi^{\circ}K||K|^{2}\leq\left(\frac{4}{3}\right)^{3}.
    $$
    Thus $$d_{3,2}(K) = \frac{|K||K^{\circ}|}{64\delta_{L}(K^{\circ})}\leq \frac{\frac{4}{3}\frac{8\sqrt{3}\pi}{2}}{64} = \frac{\pi}{4\sqrt{3}}.$$
\end{proof}

\end{section}

\section*{Acknowledgement}
\noindent I would like to thank my research advisor Nikita Kalinin for the profitable discussions about the conjecture and help with finding the right approach to solve it.

\textit{E-mail address:} \href{mailto:arkadiy.aliev@gmail.com}{arkadiy.aliev@gmail.com} 

\end{document}